\newcommand{\R}{\mathds R}

\newcommand{\N}{\mathds N}

\documentclass[twoside,a4paper,10pt]{amsart}

\usepackage{times}
\usepackage{dsfont}

\renewcommand{\contentsline}[3]{\csname new#1\endcsname{#2}{#3}}
\newcommand{\newchapter}[2]{\bigskip\hbox to \hsize{\vbox{\advance\hsize by -.5cm\baselineskip=12pt\parfillskip=0pt\leftskip=2cm\noindent\hskip -2cm #1\leaders\hbox{.}\hfil\hfil\par}$\,$#2\hfil}}
\newcommand{\newsection}[2]{\medskip\hbox to \hsize{\vbox{\advance\hsize by -.5cm\baselineskip=12pt\parfillskip=0pt\leftskip=2.5cm\noindent\hskip -2cm #1\leaders\hbox{.}\hfil\hfil\par}$\,$#2\hfil}}
\newcommand{\newsubsection}[2]{\medskip\hbox to \hsize{\vbox{\advance\hsize by -.5cm\baselineskip=12pt\parfillskip=0pt\leftskip=3.5cm\noindent\hskip -2cm #1\leaders\hbox{.}\hfil\hfil\par}$\,$#2\hfil}}

\numberwithin{equation}{section}

\title[Semi-Riemannian Bumpy Metric Theorem]{On the semi-Riemannian bumpy metric theorem}
\author[L. Biliotti]{Leonardo Biliotti}
\address{Dipartimento di Matematica\hfill\break\indent Universit\`a
di Parma\hfill\break\indent Viale G. Usberti 53/A\hfill\break\indent 43100 Parma, Italy}
\email{leonardo.biliotti@unipr.it}

\author[M. A. Javaloyes]{Miguel Angel Javaloyes}
\address{Departamento de Geometr\'{\i}a y Topolog\'{\i}a.\hfill\break\indent
 Facultad de Ciencias, Universidad de Granada.\hfill\break\indent
 Campus Fuentenueva s/n, 18071 Granada, Spain}
 \email{ma.javaloyes@gmail.es}
\thanks{The second author was partially supported by
Regional J. Andaluc\'{\i}a Grant P06-FQM-01951 and by Spanish MEC Grant MTM2007-64504.
The third author is sponsored by Capes, Brasil, Grant BEX 1509-08-0 and Fundaci\'{o}n
S\'{e}neca grant 09708/IV2/08, Spain. This manuscript was finalized
while the third author was visiting the Mathematics Department of Nara's Women University, Japan.
Special thanks are due to Prof.\ Miyuki Koiso for the invitation, and for providing an excellent
working environment in Nara.}

\author[P.\ Piccione]{Paolo Piccione}
\address{Departamento de Matem\'atica,\hfill\break\indent
Universidade de S\~ao Paulo, \hfill\break\indent Rua do Mat\~ao
1010,\hfill\break\indent CEP 05508-900, S\~ao Paulo, SP, Brazil}
\curraddr{Department of Mathematics, \hfill\break\indent
University of Murcia, Campus de Espinardo\hfill\break\indent
30100 Espinardo, Murcia, \hfill\break\indent Spain}
\email{piccione.p@gmail.com}


\date{July 21st, 2009}

\begin{document}


\theoremstyle{plain}\newtheorem*{teon}{Theorem}
\theoremstyle{definition}\newtheorem*{defin*}{Definition}
\theoremstyle{plain}\newtheorem{teo}{Theorem}[section]
\theoremstyle{plain}\newtheorem{prop}[teo]{Proposition}
\theoremstyle{plain}\newtheorem{lem}[teo]{Lemma}
\theoremstyle{plain}\newtheorem*{lem-n}{Lemma}
\theoremstyle{plain}\newtheorem{cor}[teo]{Corollary}
\theoremstyle{definition}\newtheorem{defin}[teo]{Definition}
\theoremstyle{remark}\newtheorem{rem}[teo]{Remark}
\theoremstyle{plain} \newtheorem{assum}[teo]{Assumption}
\swapnumbers
\theoremstyle{definition}\newtheorem{example}{Example}


\begin{abstract}
We prove the semi-Riemannian bumpy metric theorem using equivariant variational
genericity. The theorem states that, on a given compact manifold $M$, the set
of semi-Riemannian metrics that admit only nondegenerate closed
geodesics is generic relatively to the $C^k$-topology, $k=2,\ldots,\infty$,
in the set of metrics of a given index on $M$. A higher order genericity
Riemannian result of Klingenberg and Takens \cite{KliTak} is extended to
semi-Riemannian geometry.
\end{abstract}

\maketitle

\begin{section}{Introduction}
One of the central results in the theory of generic properties of Riemannian geodesic flows
is the so called \emph{bumpy metric theorem}. Recall that a Riemannian metric on a smooth
manifold $M$ is said to be bumpy if all its closed geodesics are nondegenerate. A closed
geodesic $\gamma$ is nondegenerate when the only periodic Jacobi fields along $\gamma$ are
constant multiples of the tangent field. The bumpy metric theorem states that bumpy metrics
over a compact manifold $M$ are \emph{generic} relatively to the $C^k$-topology, $k=2,\ldots,\infty$.
The bumpy metric theorem is attributed to Abraham, who was the first to formulate its statement
in \cite{Abr}, although the first complete proof of the result is due to Anosov, see \cite{Ano}.
The bumpy metric theorem is the keystone for other important genericity results of geodesic flows, see
for instance \cite{ContPat, KliTak}. Namely, nondegenerate geodesics are isolated, and
once one knows that the set of metrics that admit only isolated geodesics is generic, further
genericity results are obtained by showing the existence of suitable \emph{local} perturbations of
a given metric around a given closed geodesic, see Section~\ref{sec:further} of the present paper
for an example.
An important observation is that an analogous genericity result for periodic orbits
does \emph{not} hold in the more general class of Hamiltonian systems, see \cite{MeyPal}.
As pointed out by Anosov, the lack of genericity of Hamiltonian functions with only nondegerate periodic
orbits is caused by the fact that the behavior of Hamiltonian flows, unlike Riemannian geodesic flows,
depends essentially on the energy level.

A significative extension of the genericity of periodic trajectories has been proved in \cite{GonMir}
in the context of magnetic flows on a surface; the result is then used by the author to establish
an extension of the Kupka-Smale theorem.
It is a natural (and useful) question to ask whether the bumpy metric theorem can be extended
to the semi-Riemannian realm, i.e., to the case of geodesic flows of non positive definite metrics
on a given compact manifold $M$.
The main result of this paper (Theorem~\ref{thm:finalbumpy})
gives an affirmative answer to this question.
Note that, unlike the Riemannian case, when the metric is not positive definite there is
a significative qualitative change of the structure of the geodesic flow when passing from negative
to positive values of the energy.

There are two ways of characterizing nondegeneracy of closed geodesics, corresponding to the \emph{dynamical}
and the \emph{variational} approach\footnote{%
The dynamical approach consists in studying closed geodesics as those geodesics that are
fixed points for the Poincar\'e map. The variational approach consists in studying geodesics as critical
points of the geodesic action functional defined in the free loop space.}
to the closed geodesic problem, that can be described as follows. Denote by $T^1M$ the unit tangent bundle of $M$ relatively to a Riemannian
metric on $M$; for all $v\in T^1M$, let $\gamma_v:\left[0,+\infty\right[\to M$ be the
unique geodesic with $\dot\gamma(0)=v$.

From the dynamical viewpoint, nondegeneracy of a closed geodesic
$\gamma_0:\left[0,+\infty\right[\to M$ of period $T$
means that the map $\R^+\times T^1M\ni(t,v)\mapsto \big(v,\dot\gamma_v(t)\big)\in T^1M\times T^1M$
is transverse to the diagonal of $T^1M\times T^1M$ at the point $(T,v_0)$, where $v_0=\dot\gamma_0(0)$.
Anosov's proof of the Riemannian bumpy metric theorem uses this approach, and it employs the transversality theorem.

The dynamical approach does not work well when considering semi-Riemannian metrics, starting from
the observation that even the notion of unit tangent bundle itself is not very meaningful in semi-Riemannian
geometry. Distinguishing \emph{causal} notions of ``unit tangent bundles'', i.e., timelike, lightlike
and spacelike, is also not very meaningful when dealing with families of metrics.
For our proof of the semi-Riemannian bumpy metric theorem we will use the variational approach;
from this viewpoint, nondegeneracy for a closed geodesic $\gamma$ means that $\gamma$ is a nondegenerate
critical point of the geodesic action functional, in the sense that the kernel of the second
variation of the functional, the so-called \emph{index form} of the geodesic, has one dimensional
kernel consisting of the constant multiples of the tangent field $\dot\gamma$.
Using an idea of White \cite{Whi} (see also \cite{BilJavPic}), semi-Riemannian metrics with degenerate
geodesics are characterized as critical values of a certain nonlinear Fredholm
map between infinite dimensional Banach spaces, and a genericity result in this situation can be obtained
as an application of Sard--Smale theorem.
However, compared to the fixed endpoint case studied in \cite{BilJavPic, GiaGiaPicCMP},
the periodic case has two important differences. First, the geodesic action functional is invariant
by the action of $\mathds S^1$ obtained by rotation of the parameter space (in fact, the entire orthogonal
group $\mathrm O(2)$ acts on the free loop space, but here it will suffice to consider the restriction
of this action to its identity connected component).
Second, in the periodic case an essential transversality condition employed in the application of
Sard--Smale theorem only holds for \emph{prime} closed geodesics, i.e., geodesics that are not obtained
as $n$-fold iteration of some other geodesic, with $n>1$.

In order to deal with the question of rotations, we will prove an abstract equivariant genericity result
(Theorem~\ref{thm:abstrgenericitysmooth})
in the case of actions of compact Lie groups all of whose orbits have the same dimension. Note that the
$\mathds S^1$-action on $\Lambda$ has only finite cyclic groups as stabilizers, hence
all its orbits are homeomorphic to $\mathds S^1$ itself. This genericity result uses the notion of
\emph{good submanifolds} of manifolds endowed with group actions. Roughly speaking, every second countable
Hilbert manifold endowed with the smooth action of a compact Lie group $G$ with equidimensional
orbits contains a countable family $(S_n)_{n\in\N}$, of submanifolds that are everywhere transverse to the orbits,
with codimension equal to the dimension of the orbits, and such that every point of the manifold
belongs to the orbit of some point of some $S_n$. In this situation, given a $G$-invariant function
$f$, then all its critical orbits are nondegenerate in the equivariant sense if and only if the
restriction $f\vert_{S_n}$ is a Morse function for all $n$ (Proposition~\ref{thm:MorseGMorse}).

A somewhat annoying point is that, with the above formulation, the abstract equivariant genericity result
cannot be applied directly to the $\mathds S^1$-action on the free loop space, in that such action
is not smooth, but only continuous. Due to this reason, the existence of good submanifolds for the
closed geodesic problem has to be shown with direct arguments (see Subsection~\ref{sub:goodsubmanifolds}),
and this is one of the most original contributions of the present paper.
Using the existence of good submanifolds, the abstract equivariant genericity problem is then
applied to the open subset $\widetilde\Lambda$ of $\Lambda$ consisting of prime curves, obtaining
a weak version of the bumpy metric theorem (Proposition~\ref{thm:weakbumpy}).
In Appendix~\ref{sec:appA} we will use the existence of good submanifolds for the free loop
space to prove a result of smooth dependence of closed geodesics, needed for our theory,
and that seems to have an interest in its own.

Genericity of nondegeneracy of iterates does not follow from the equivariant variational setup
due to a subtle technical problem, as discussed in details in \cite{BilJavPic}. Namely, the transversality
assumption employed in the abstract genericity result is not satisfied at a class of \emph{very degenerate}
iterates of closed geodesics.
In order to deal with iterates, we will follow Anosov's ingenious idea in \cite{Ano}, with
suitable modifications that make it work also in the semi-Riemannian case (Subsection~\ref{sub:iterate}).
We introduce families of metrics $\mathcal M(a,b)$ parameterized by two positive real numbers $a,b$,
that correspond to the \emph{period} and to the \emph{minimal period} of closed geodesics.
For our semi-Riemannian extension, the notion of period (which is meaningless in the case
of lightlike geodesics) is replaced by the notions of \emph{energy} relatively to an
auxiliary Riemannian metric. The set of bumpy metrics correspond to the countable intersection
$\bigcap_{n\ge1}\mathcal M(n,n)$, and a proof of the bumpy metric theorem is obtained
by showing that each $\mathcal M(a,b)$ is open and dense in the set of metrics.
As in Anosov's paper, the crucial step of the proof (Lemma~\ref{thm:lemit5}) uses a local
metric perturbation argument due to Klingenberg. This argument works in the semi-Riemannian
case only for non lightlike geodesics; for the lightlike case a further perturbation
argument is needed (Proposition~\ref{thm:prop2} and Corollary~\ref{thm:cor3}).

By the above arguments, genericity of the set of bumpy metrics is proved in the $C^k$-topology,
with $k\ge2$. Genericity in the $C^\infty$-topology, which does not come from a Banach space
setup, is obtained by a standard argument. For the reader's convenience, this is discussed
in Appendix~\ref{sec:appB}.

Finally, in Section~\ref{sec:further} we will show how to extend to semi-Riemannian metrics
another important genericity result, which in the Riemannian case is due to Klingenberg
and Takens \cite{KliTak}. Such result states that one has $C^{k+1}$-genericity of those metrics for
which the Poincar\'e map of every closed geodesic has $k$-th jet that belongs to a prescribed
open dense invariant subset of the space of $k$-jets of symplectic maps.
This generalizes the bumpy metric theorem, in that the bumpy condition is an open and
 dense condition on the first jet of the Poincar\'e map (eigenvalues different from $1$ of the linearized
 Poincar\'e map). In the semi-Riemannian context, the statement of this theorem makes sense only
 for metrics that do not possess closed lightlike geodesics. Once the bumpy metric theorem
 has been established, Klingenberg-Takens genericity result is obtained by the same perturbation
 argument employed in the Riemannian case, together with a result of genericity of semi-Riemannian metrics
 without closed lightlike geodesics (Proposition~\ref{thm:Laopen}).

\end{section}

\begin{section}{Genericity of equivariant nondegeneracy}
\label{sec:equivnondegeneracy}
Let us recall the following result from \cite{BilJavPic, Chi80}:
\begin{teo}\label{thm:abstrgenericity}
Let $F:A\to\R$ be a function of class $C^2$ defined on an open subset of the product $X\times Y$, where
$X$ is a separable Banach manifold and $Y$ is a separable Hilbert manifold. Assume that for every $(x_0,y_0)\in A$
such that $\frac{\partial f}{\partial y}(x_0,y_0)=0$ the following conditions hold:
\begin{itemize}
\item[(a)] the second derivative $\frac{\partial^2f}{\partial y^2}(x_0,y_0)$ is a (self-adjoint) Fredholm operator
on $T_{y_0}Y$;
\item[(b)] for all $v\in\mathrm{Ker}\big[\frac{\partial^2f}{\partial y^2}(x_0,y_0)\big]\setminus\{0\}$ there
exists $w\in T_{x_0}X$ such that \[\frac{\partial^2f}{\partial x\partial y}(x_0,y_0)\big(v,w\big)\ne0.\]
\end{itemize}
Denote by $\Pi:X\times Y\to X$ the projection, and, for $x\in\Pi(A)$, set \[A_x=\big\{y\in Y:(x,y)\in A\big\}.\]
Then, the set of $x\in X$ such that the functional $A_x\ni y\mapsto f(x,y)\in\R$ is a Morse function is generic
in $\Pi(A)$.
\end{teo}
\begin{proof}
The statement above is \cite[Corollary~3.4]{BilJavPic}; a more general statement is due to Chillingworth,
see \cite{Chi80}.
\end{proof}

\noindent Assumption (b) in Theorem~\ref{thm:abstrgenericity} implies the transversality of the map
$\frac{\partial f}{\partial y}:A\to TY^*$ to the zero section of the cotangent bundle $TY^*$;
we will often refer to it as the \emph{transversality condition}. It guarantees that the set:
\[\mathfrak M=\Big\{(x_0,y_0)\in A:\tfrac{\partial f}{\partial y}(x_0,y_0)=0\Big\}\]
is a submanifold of $A$, and that the restriction of the projection $\Pi:\mathfrak M\to X$ is a nonlinear
Fredholm operator of index $0$. The regular values of this map are precisely the points
$x\in X$ such that the functional $A_x\ni y\mapsto f(x,y)\in\R$ is a Morse function, and the proof
of Theorem~\ref{thm:abstrgenericity} follows from the infinite dimensional version of Sard's theorem
\cite{Sma}.
\smallskip

We want to formulate an analogous result for smooth functions that are invariant by the action
of a Lie group.
\subsection{$G$-Morse functions}
Let us assume the following setup. Let $Y$ be a Hilbert manifold, $G$ a finite dimensional Lie group acting on $Y$, and let $\mathfrak g$ be the Lie algebra of $G$.
We will denote by $g\cdot y$ the action
of the element $g\in G$ on $y\in Y$; in this section, the action of $G$ on $Y$ will be assumed to be of class
$C^1$. For all $y\in Y$, let $\beta_y:G\to Y$ be the map of class $C^1$ defined by
$\beta_y(g)=g\cdot y$. The image of the differential $\mathrm d\beta_y(1):\mathfrak g\to T_yY$ is a
(finite dimensional) subspace $D_y$ of $T_yY$, which is the tangent space to the \emph{orbit} $G\cdot y$ at the
point $y$. If the dimension of $D_y$ does not depend on $y$, then $D=\{D_y\}_{y\in Y}$ is a smooth distribution on $Y$.
\begin{defin}
A submanifold $S\subset Y$ is said to be \emph{transverse to $D$ at $y\in S$} if $T_yY$ is the \emph{direct sum}
$T_yS\oplus D_y$. The submanifold $S$ is transverse to $D$ if it is transverse to $D$ at every $y\in S$.
\end{defin}
One proves easily the following:
\begin{lem}\label{thm:lem1}
Assume that the dimension of $D_y$ does not depend on $y$. Let $S$ be a submanifold of $Y$ which is
transverse to $D$ at some $y\in S$. Then:
\begin{itemize}
\item[(a)] there exists an open submanifold $S'\subset S$ containing $y$ which
is transverse to $D$;
\item[(b)] the set $G\cdot S=\big\{g\cdot s:g\in G,\ s\in S\big\}$ contains an open neighborhood of $y$ in $Y$.
\end{itemize}
\end{lem}
\begin{proof}
Part (a) follows readily from the differentiability of the action, observing that the condition of being
transverse to a continuous distributions of spaces of constant dimension is open. Part (b) follows
from the inverse function theorem applied to the function $G\times S\ni(g,s)\mapsto g\cdot s\in Y$ at
the point $(1,y)$.
\end{proof}
\begin{cor}\label{thm:countablefamsubmanifolds}
In the above setup, assume that $Y$ is second countable, and that the dimension of $D_y$ does
not depend on $y$. Then, there exists a countable family $S_n$, $n=1,\ldots$, of submanifolds of $Y$
that are transverse to $D$ and such that every orbit of $G$ intercepts some $S_n$.
\end{cor}
\begin{proof}
By (a) of Lemma~\ref{thm:lem1}, every $y\in Y$ is contained in a submanifold $S_y$ of $Y$ which is transverse
to $D$. By (b) of Lemma~\ref{thm:lem1}, there exists an open neighborhood $U_y$ of $y$ in $Y$ such that every point
in $U_y$ belongs to the orbit of some element of $S_y$. The open covering $\{U_y\}_{y\in Y}$ of $Y$
admits a countable subcover $U_{y_1},\ldots,U_{y_n},\ldots$, by the assumption of second countability.
The family $S_n$ is given by $S_{y_n}$, $n=1,\ldots$.
\end{proof}

Let now $f:Y\to\R$ be a smooth $G$-invariant function.
\begin{defin}
Given a critical point $y$ of $f$, we say that $f$ is  \emph{$G$-nondegenerate at $y$} if
the Hessian $\mathrm H^f(y)$ of $f$ at $y$ is (strongly) nondegenerate when restricted
to some closed complement of $D_y$. We will say that $f$ is \emph{$G$-Morse} if it is $G$-nondegenerate at
every critical point $y$.
\end{defin}
It is easy to see that the nondegeneracy above does not depend on the choice of a complement of $D_y$.
Also, by $G$-invariance, it suffices to check the $G$-Morse property only at one point of each critical
orbit of $f$.
\begin{prop}\label{thm:MorseGMorse}
Under the assumption of Corollary~\ref{thm:countablefamsubmanifolds}, $f$ is $G$-Morse if and only if
$f\vert_{S_n}$ is a Morse function for all $n$.
\end{prop}
\begin{proof}
First, we observe that $y\in S_n$ is a critical point of $f$ if and only if it is a critical point
for the restriction $f\vert_{S_n}$. This is because, by $G$-invariance, $\mathrm df(y)$ vanishes on the
space $\mathcal D_y$, which is complementary to $T_yS_n$. Moreover, the Hessian $\mathrm H^{f\vert_{S_n}}(y)$  at $y$
of the restriction $f\vert_{S_n}$ is the restriction to $T_yS_n\times T_yS_n$ of the Hessian
$\mathrm H^f(y)$.
Thus, if $f$ is $G$-Morse and $y\in S_n$ is a critical point of $f$, then $\mathrm H^{f\vert_{S_n}}(y)$
is strongly nondegenerate, i.e., $f\vert_{S_n}$ is a Morse function. Conversely, assume that $f\vert_{S_n}$ is a
Morse function, and let $y$ be a critical point of $f$. As observed above, one is free to choose
any point in the orbit of $y$, and by Corollary~\ref{thm:countablefamsubmanifolds}
it follows that it suffices to consider the case that $y\in S_n$ for some $n$. Then, $\mathrm H^f(y)$ is
strongly nondegenerate when restricted to $T_yS_n$, which is complementary to $D_y$, hence $f$ is
$G$-Morse.
\end{proof}
\subsection{Genericity of the $G$-Morse condition}
We will now assume that $X$ is a separable Banach manifold, $A\subset X\times Y$ is an open subset which
is invariant by the action of $G$ on the second variable, i.e., $(x,y)\in A$ implies $(x,g\cdot y)\in A$ for
all $g\in G$. Let $f:A\to\R$ be a function of class $C^k$, with $k\ge2$, which is $G$-invariant in the second variable,
i.e., it is such that $f(x,g\cdot y)=f(x,y)$ for all $(x,y)\in A$ and all $g\in G$.
For $x\in X$, let $A_x\subset Y$ be the open subset:
\[A_x=\big\{y\in Y:(x,y)\in A\big\};\]
clearly, $A_x$ is $G$-invariant.
We want to study the set of $x\in X$ such that the smooth map:
\[f_x=f(x,\cdot):A_x\longrightarrow\R\]
is $G$-Morse.
\begin{teo}\label{thm:abstrgenericitysmooth}
In the above setup, assume the following:
\begin{itemize}
\item[(a)] for all $(x_0,y_0)$ such that $y_0\in A_{x_0}$ is a critical point of $f_{x_0}$, the Hessian
$\mathrm H^{f_{x_0}}(y_0)$ is Fredholm;
\item[(b)] for all $(x_0,y_0)$ such that $y_0\in A_{x_0}$ is a critical point of $f_{x_0}$ and for
all $v\in\mathrm{Ker}\big(\mathrm H^{f_{x_0}}(y_0)\big)\setminus D_{y_0}$ there exists $w\in T_{x_0}X$
such that
\begin{equation}\label{eq:transversality}
\frac{\partial^2f}{\partial x\partial y}(x_0,y_0)\big(v,w\big)\ne0.
\end{equation}
\end{itemize}
Denote by $\Pi:X\times Y\to X$ the projection.
Then, the set of $x\in X$ such that the functional $A_x\ni y\mapsto f(x,y)\in\R$ is a $G$-Morse function is generic
in $\Pi(A)$.
\end{teo}
\begin{proof}
Consider the countable family $S_n$ of submanifolds of $X$ as in Corollary~\ref{thm:countablefamsubmanifolds}.
Let $A_n$ be the open subset of $X\times S_n$ defined by:
\[A_n=\big\{(x,s)\in X\times S_n:\exists\,g\in G\ \text{with}\ (x,g\cdot s)\in A\big\};\]
we will apply Theorem~\ref{thm:abstrgenericity} to the restriction $f_n$ of $f$ to $A_n$.
First observe that given $(x_0,y_0)\in A_n$ such that $\frac{\partial f_n}{\partial y}(x_0,y_0)=0$,
then the Hessian $\mathrm H^{(f_n)_{x_0}}(y_0)$ is Fredholm, because it is the restriction of
a Fredholm symmetric bilinear form to a finite codimensional space.

Now, given $v\in\mathrm{Ker}\big[\mathrm H^{(f_n)_{x_0}}(y_0)\big]\ne0$, then
$v\in\mathrm{Ker}\big(\mathrm H^{f_{x_0}}(y_0)\big)\setminus D_{y_0}$, and thus there exists $w\in T_{x_0}X$
such that $\frac{\partial^2f}{\partial x\partial y}(x_0,y_0)\big(v,w\big)\ne0$. Observe that
$\frac{\partial f}{\partial y}(x,y_0):T_{y_0}Y\to\R$ is identically zero on $D_{y_0}$ and its restriction
to $T_{y_0}S_n$ coincides with $\frac{\partial f_n}{\partial y}(x,y_0)$. Thus, for all $w\in T_{x_0}X$,
the second derivative
$\frac{\partial^2f}{\partial x\partial y}(x_0,y_0)\big(\cdot,w\big)$ vanishes identically on $D_{y_0}$, and its
restriction to $T_{y_0}S_n$ coincides with $\frac{\partial^2f_n}{\partial x\partial y}(x_0,y_0)\big(\cdot,w\big)$.

It follows that  $\frac{\partial^2f_n}{\partial x\partial y}(x_0,y_0)\big(v,w\big)\ne0$, hence
Theorem~\ref{thm:abstrgenericity} applies to each $f_n$.

\noindent
Denote by $B_n$ the subset of $X$ consisting of those $x$'s such that the functional $y\to f_n(x,y)$ is Morse.
By Proposition~\ref{thm:MorseGMorse}, the set of $x\in X$ such that $f$ is $G$-Morse is the countable intersection
$B=\bigcap_{n\ge1}B_n$, and each $B_n$ is generic in $\Pi(A)$, hence $B$ is generic in $\Pi(A)$. This concludes the proof.
\end{proof}
\end{section}
\begin{section}{The closed geodesic problem}
\subsection{Nondegenerate closed geodesics}
Let us now consider the closed geodesic problem in semi-Riemannian manifolds.
Let $M$ be a differentiable manifold, and let $\mathbf g$ be a semi-Riemannian metric
tensor on $M$.
A closed geodesic $\gamma:\mathds S^1\to M$ is \emph{nondegenerate} when there is no periodic
Jacobi field $J$ along $\gamma$ which is not a constant multiple of the tangent field
$\dot\gamma$. Equivalently, in the language of equivariant functions,
$\gamma$ is nondegenerate if it is a $G$-nondegenerate critical point of the geodesic action
functional $f_{\mathbf g}(y)=\frac12\int_{\mathds S^1}{\mathbf g}(\dot y,\dot y)\,\mathrm d\theta$, defined in
the Hilbert manifold $\Lambda$ of all closed curves $y:\mathds S^1\to M$ of Sobolev class $H^1$, and
$G=\mathds S^1$ is the circle, acting by right composition $(g\cdot y)(\theta)=y(g\theta)$.

Using the theory developed in Section~\ref{sec:equivnondegeneracy}, we want to study the genericity
of the set of all semi-Riemannian metrics ${\mathbf g}$ (having fixed index) that do not have
degenerate closed geodesics. Equivalently, this is the set of all metrics ${\mathbf g}$ whose
associated geodesic action functional on $\Lambda$ is $G$-Morse. Following the classical literature
(see \cite{Abr, Ano}), a metric tensor whose geodesic action functional is $G$-Morse will be called
a \emph{bumpy metric}.
However, the closed geodesic problem does not quite fit in the theory of Section~\ref{sec:equivnondegeneracy},
due to two different obstructions. First, there is a question of
regularity, in that \emph{the $\mathds S^1$-action on $\Lambda$ is not differentiable},
but only continuous. In particular, this implies that the construction of the sequence of \emph{good}
manifolds $S_n$ of Corollary~\ref{thm:countablefamsubmanifolds} does not work in this case.
Second, there is a problem with the transversality condition (b) of Theorem~\ref{thm:abstrgenericitysmooth},
which \emph{does not hold} in general for a class of closed geodesics that are obtained by iteration of
a closed geodesic.

We will first study the genericity of the nondegeneracy problem for \emph{prime} geodesics, i.e.,
geodesics that are not iterates, and then we will study the problem of iterates.
The question of lack of smoothness is studied with an alternative construction of special
submanifolds $S_n$ in the free loop space $\Lambda$.
\subsection{Good submanifolds of $\Lambda$}
\label{sub:goodsubmanifolds}
Recall that for $y\in\Lambda$, the tangent space $T_y\Lambda$ is identified with the Hilbert
space of all periodic vector fields of Sobolev class $H^1$ along $y$.
We will temporarily employ the following notations. Denote by $\Lambda^2$ the dense
subset of $\Lambda$ consisting of all curves of Sobolev class $H^2$. For $y\in\Lambda^2$,
let $D_y\subset T_y\Lambda$ denote the $1$-dimensional space spanned by the tangent field
$\dot y$. Let $\mathbf g$ be a fixed semi-Riemannian metric tensor in $M$, and denote by $f_{\mathbf g}:\Lambda\to\R$
the corresponding geodesic action functional.
\begin{defin}
A submanifold $S\subset\Lambda$ is \emph{good} for $\mathbf g$ if:
\begin{itemize}
\item given $y\in S$, if $y$ is a critical point of the restriction $f_{\mathbf g}\vert_S$
then $y$ is a critical point of $f_{\mathbf g}$ (i.e., $y$ is a closed $\mathbf g$-geodesic,
in particular, $y$ is of class $H^2$);
\item if $y\in S$ is a critical point of $f_{\mathbf g}$, then
$T_y\Lambda$ is direct sum of $T_yS$ and $D_y$.
\end{itemize}
\end{defin}
We will show the existence of sufficiently many good submanifolds of $\Lambda$ for all metrics $\mathbf g$; more precisely:
\begin{prop}\label{thm:exSngeo}
There exists a sequence $S_n$ of submanifolds of $\Lambda$ and an open subset $\mathcal A$ of $\Lambda$
such that:
\begin{itemize}
\item[(a)]  $\mathcal A$ contains the set of all closed curves in $\gamma:\mathds S^1\to M$ of class $C^2$;
\item[(b)] for all $n$ and for all semi-Riemannian metric $\mathbf g$ on $M$, $S_n$ is good for $\mathbf g$;
\item[(c)] for all $y\in\mathcal A$, the $\mathds S^1$-orbit of $y$ intercepts some $S_n$.
\end{itemize}
\end{prop}
\begin{proof}
First, we will show that through every $\gamma\in\Lambda$ of class $C^2$
there exists a submanifold $S_\gamma$ of $\Lambda$ which is good for all $\mathbf g$
(in fact, our construction will not depend on the choice of $\mathbf g$).
Let $\mathbf g_{\mathrm R}$ be a fixed auxiliary Riemannian metric on $M$, and let $\exp$
denote its exponential map; the Hilbert structure on all vector bundles involved in this proof
use the metric $\mathbf g_{\mathrm R}$.
Denote by $\mathcal E\to\Lambda$ the \emph{mixed} vector bundle over $\Lambda$, whose
fiber at the point $y\in\Lambda$ is the Hilbert space of all vector fields of class $L^2$
along $y$.  There is a natural continuous inclusion $T\Lambda\hookrightarrow\mathcal E$, and
$D=(D_y)_{y\in\Lambda^2}$ extends to a continuous subbundle of $\mathcal E$.
Let $\gamma\in\Lambda$ be of class $C^2$ and let $\varphi:U\to T_\gamma\Lambda$ be the chart defined on
an open neighborhood $U$ of $\gamma$ in $\Lambda$ and taking values in an open neighborhood
of $0$ in the Hilbert space $T_\gamma\Lambda$ which is constructed via exponential map.
More precisely, given a vector field $V$ of class $H^1$ along $\gamma$, near $0$, the inverse
image $z=\varphi^{-1}(V)$ is the $H^1$-curve defined by $z(\theta)=\exp_{\gamma(\theta)}\big(V(\theta)\big)$
for all $\theta\in\mathds S^1$. Using the chart $\varphi$, one has a trivialization of $\mathcal E\vert_U$
as follows: for all $z\in U$, the differential $\mathrm d\varphi(z):T_z\Lambda\to T_\gamma\Lambda$
extends continuously to an isomorphism of $\mathcal E_z$ onto $\mathcal E_\gamma$. Thus, we have
a vector bundle isomorphism $\mathcal E\vert_U\cong U\times\mathcal E_\gamma$.

A good submanifold $S_\gamma$ through $\gamma$ can now be chosen as follows. Let $H_0$ be a hyperplane
of $T_\gamma\Lambda$ which is closed in $T_\gamma\Lambda$ relatively to the $L^2$-topology, and that
does not contain $\dot\gamma$. For instance, $H_0$ can be chosen as the $L^2$-orthogonal space to
$\dot\gamma$ in $T_\gamma\Lambda$.
Define $S_\gamma$ as the inverse image $\varphi^{-1}(V_0)$ of a small neighborhood $V_0$ of $0$
in $H_0$. The function $z\mapsto T_zS_\gamma$ defined on $S_\gamma$ is continuous relatively to the $H^1$-topology,
regarding $T_zS_\gamma$ as a subspace of $\mathcal E_z$. Similarly, $z\mapsto D_z$ is continuous
relatively to the $H^1$-topology when $D_z$ is seen as a subspace of $\mathcal E_z$.
Thus, the transversality of $T_zS_\gamma$ and $D_z$ is maintained for $z\in\Lambda^2\cap S_\gamma$ near $\gamma$
in the $H^1$-topology.

Let us show that $S_\gamma$ is good. Assume that $y\in S_\gamma$ is a critical point of $f_{\mathbf g}\vert_{S_\gamma}$.
Since $T_yS_\gamma$ is $L^2$-closed in $T_y\Lambda$, then the linear operator $\mathrm df_{\mathbf g}(y)$ on $T_y\Lambda$,
given by:
\[\mathrm df_{\mathbf g}(y)V=\int_{\mathds S^1}\mathbf g(\mathbf DV,\dot y)\,\mathrm d\theta\]
is bounded in the $L^2$-topology, because it vanishes on an $L^2$-closed finite codimensional subspace
of $T_y\Lambda$ (here, $\mathbf DV$ is the covariant derivative of $V$ along $y$ relatively to the
Levi--Civita connection of $\mathbf g$).
By a standard boot-strap argument, it follows easily that $y$ is a curve of class $H^2$; then, since $T_y\Lambda=T_yS_\gamma\oplus D_y$,
and $\mathrm df_{\mathbf g}(y)$ vanishes on $D_y$, it follows that $y$ is a critical point
of $f_{\mathbf g}$ in $\Lambda$.
Hence, $S_\gamma$ is good.

Now, we observe that since the orbit $\mathds S^1\cdot\gamma$ intercepts
transversally $S_\gamma$, then by continuity all orbits sufficiently near $\mathds S^1\cdot\gamma$ also
intercept $S_\gamma$. By possibly reducing the size of $S_\gamma$, we can therefore assume that $\mathds S^1\cdot S_\gamma$
contains an open neighborhood $U_\gamma$ of $S_\gamma$ in $\Lambda$.
Let $\mathcal A$ be the open subset of $\Lambda$ given by the union of all $U_\gamma$, as $\gamma$ runs in the
set of curves of class $C^2$ in $\Lambda$. Since $\Lambda$ is second countable, then also $\mathcal A$ is second countable,
and thus there exists a countable subfamily $U_{\gamma_1},U_{\gamma_2},\ldots,U_{\gamma_n},\ldots$ whose union
is equal to $\mathcal A$. The desired countable family of good submanifolds is given by $S_n=S_{\gamma_n}$, $n\ge1$.
\end{proof}
Thus, we have the analogue of Proposition~\ref{thm:MorseGMorse} for geodesic functionals:
\begin{cor}
A semi-Riemannian metric $\mathbf g$ on $M$ is bumpy if and only if the restriction of the geodesic
action functional $f_{\mathbf g}$ to each submanifold $S_n$ as in Proposition~\ref{thm:exSngeo} is a Morse function.\qed
\end{cor}
\subsection{Prime geodesics}
Let us introduce some terminology. The \emph{stabilizer} of a non constant
closed curve $y:\mathds S^1\to M$, i.e., the set of $g\in\mathds S^1$ such that $y(g\theta)=y(\theta)$
for all $\theta\in\mathds S^1$, denoted by $\mathrm{stab}(y)$, is easily seen to be a finite cyclic
group of $\mathds S^1$. The orbit of $y$, $\mathds S^1\cdot y$,
which is homeomorphic to the quotient $\mathds S^1/\mathrm{stab}(y)$,
is therefore homeomorphic to $\mathds S^1$; when $y$ is a curve of class $C^2$, then
$\mathds S^1\cdot y$ is a $C^1$-submanifold of $\Lambda$.

Let us say that a curve $y\in\Lambda$ is \emph{prime} if $\mathrm{stab}(y)$ is trivial, i.e., if
$y$ is not the iterate of some other closed curve in $M$. The subset $\widetilde\Lambda\subset\Lambda$
consisting of prime curves is open in $\Lambda$, and it is clearly invariant by the action
of $\mathds S^1$.
\begin{prop}[Weak semi-Riemannian bumpy metric theorem]
\label{thm:weakbumpy}
Let $M$ be a compact manifold, $k\ge2$, and let $i$ be an integer in $\big\{0,\ldots,\mathrm{dim}(M)\big\}$.
Denote by $\mathrm{Met}(M,i;k)$ the set of all semi-Riemannian metric tensor of class $C^k$
on $M$ having index equal to $i$, endowed with the topology of $C^k$ convergence.
The subset of $\mathrm{Met}(M,i;k)$ consisting of all metric tensors all of whose prime geodesics are
nondegenerate is generic.
\end{prop}
\begin{rem}
It should be observed that $\mathrm{Met}(M,i;k)$ is non empty under some assumptions on the topology of
$M$. More precisely, $\mathrm{Met}(M,i;k)\ne\emptyset$ if and only if $M$ admits a $C^k$-distribution of
rank $i$. Namely, assume such a distribution $\Delta\subset TM$ is given, and let $\mathbf g_{\mathrm R}$ be an arbitrary
smooth Riemannian metric on $M$. Then, a semi-Riemannian metric $\mathbf g$ of index $i$ on $M$ can be defined
by setting $\mathbf g=\mathbf g_{\mathrm R}$ on the $\mathbf g_{\mathrm R}$-orthogonal complement $\Delta^\perp$ of $\Delta$,
$\mathbf g=-\mathbf g_{\mathrm R}$ on $\Delta$, and by setting $\mathbf g(v,w)=0$ for $v\in\Delta$ and $w\in\Delta^\perp$ $\mathbf g$-orthogonal.
Conversely, assume that a $C^k$-metric $\mathbf g$ of index $i$ is given, and let $T$ be the unique
$\mathbf g_{\mathrm R}$-symmetric $(1,1)$  tensor on $M$ such that $\mathbf g=\mathbf g_{\mathrm R}(T\cdot,\cdot)$.
Then, a distribution of rank $i$ on $M$ is given by considering at each point $p\in M$ the direct sum of the
eigenspaces of $T_p$ relative to negative eigenvalues. A standard functional analytical argument\footnote{%
Regularity of such distribution follows readily from the smoothness of the function that associates
to each symmetric matrix $A$ the orthogonal projection $P_A^-$ onto the direct sum of its negative eigenspaces.
This smoothness is proved using a formula that gives $P_A^-$ as a line integral:
let $\gamma_j$, $j=1,\ldots,i$ be smooth curves in the complex plane that make small circles
around the negative eigenvalues of $A$ oriented counterclockwise. For $A'$ near $A$, the orthogonal
projection $P_{A'}^-$ can be written as the line integral
$P_{A'}^- = \sum_{j=1}^i \frac1{2\pi\mathfrak i} \oint_{\gamma_j}(z-A')^{-1}\,\mathrm dz$.}
shows that the distribution obtained in this way has the same regularity as the metric $\mathbf g$.
It is known that, for an orientable manifold $M$, the existence of a smooth distribution of rank one
(a line field), is equivalent to the vanishing of the Euler class of $M$.
Except for the case $i=1$, it is in general a quite difficult task to
give a universal necessary and sufficient condition for the existence of distributions of rank $i$ on $M$.
\end{rem}
\begin{proof}[Proof of Proposition~\ref{thm:weakbumpy}]
Consider the $C^2$ function $f:\mathrm{Met}(M,i;k)\times\widetilde\Lambda\to\R$ defined by:
\[f(\mathbf g,y)=f_{\mathbf g}(y)=\tfrac12\int_{\mathds S^1}\mathbf g(\dot y,\dot y)\,\mathrm d\theta.\]
A point $(\mathbf g_0,y_0)$ is such that $\frac{\partial f}{\partial\mathbf g}(\mathbf g_0,y_0)=0$
if and only if $y_0$ is a prime closed $\mathbf g_0$-geodesic in $M$. Here, $\mathrm{Met}(M,i;k)$
is seen as an open subset of the Banach space of all $(0,2)$-symmetric tensors of class $C^k$ on $M$,
endowed with the $C^k$-topology.
Clearly, $f$ is invariant by the action of $G=\mathds S^1$ on the second variable but, as observed
above, such action is only continuous.

Given $(\mathbf g_0,y_0)$ such that $\frac{\partial f}{\partial\mathbf g}(\mathbf g_0,y_0)=0$, the
Hessian $\mathrm H^{f_{\mathbf g_0}}(y_0)=\frac{\partial^2f}{\partial y^2}(\mathbf g_0,y_0)$ is identified
with the standard index form of the closed geodesic $y_0$, given by the symmetric bilinear form
defined in the Hilbert space of periodic vector fields of class $H^1$ along $y_0$:
\[I_{\mathbf g_0,y_0}(V,W)=\int_{\mathds S^1}\mathbf g_0\big(V',W')+\mathbf g_0(R_0(\dot y_0,V)\,\dot y_0,W\big)\,\mathrm d\theta,\]
where $R_0$ is the curvature tensor\footnote{%
The curvature tensor $R$ of a connection $\nabla$ is chosen with
the sign convention: $R(X,Y)=[\nabla_X,\nabla_Y]-\nabla_{[X,Y]}$.}
of $\mathbf g_0$ and the superscript prime means covariant derivative of
vector fields along $y_0$ induced by the Levi--Civita connection of $\mathbf g_0$.
This is a Fredholm bilinear form, being a compact perturbation of the strongly nondegenerate
symmetric bilinear form:
\[(V,W)\longmapsto\int_{\mathds S^1}\mathbf g_0\big(V',W')\,\mathrm d\theta.\]
Let us now consider the manifold $\widetilde\Lambda$ of all prime curves in $\Lambda$, and the continuous
action of the group $G=\mathds S^1$.
Using the countable family of good submanifolds $S_n$ constructed\footnote{%
Observe that the statement
of Proposition~\ref{thm:exSngeo} holds if one replaces $\Lambda$ with $\widetilde\Lambda$.}
in Proposition~\ref{thm:exSngeo},
the genericity result of Theorem~\ref{thm:abstrgenericitysmooth}
applies to the geodesic action functional $f$ applied to the Banach space $X$ of all symmetric
$(0,2)$-tensors of class $C^k$ on $M$, $Y=\Lambda$, and $A=\mathrm{Met}(M,i;k)\times\widetilde\Lambda$, once
we show that the transversality condition \eqref{eq:transversality} holds.

In our geodesic setup, condition \eqref{eq:transversality} reads as follows (see \cite[Section~4]{BilJavPic}
for details). Given a metric $\mathbf g_0$ of class $C^k$ in $M$, a prime degenerate $\mathbf g_0$-geodesic $\gamma_0$
in $M$, and a nontrivial Jacobi field $J_0$ along $\gamma_0$ (here, non trivial means that $J_0$ is not
a constant multiple of the tangent field $\dot\gamma_0$), there must exist a symmetric $(0,2)$-tensor
$h$ of class $C^k$ in  $M$ such that the following holds:
\begin{equation}\label{eq:transvgeo}
\int_{\mathds S^1}\left(h(\dot\gamma_0,J_0')+\tfrac12\nabla h(J_0,\dot\gamma_0,\dot\gamma_0)\right)\mathrm d\theta\ne0.
\end{equation}
Here, $\nabla$ is an arbitrarily fixed symmetric connection in $M$ (the left hand side of \eqref{eq:transvgeo}
does not depend on the choice of the connection), and $J_0'$ is the covariant derivative of $J_0$ along
$\gamma_0$ relatively to such connection. Given $\gamma_0$ and $J_0$ as above, in order to construct
$h$ for which \eqref{eq:transvgeo} holds, one first observe that the set of instants $\theta\in\mathds S^1$
at which $\dot\gamma_0(\theta)$ is multiple of $J_0(\theta)$ is finite. Moreover, since $\gamma_0$ is prime,
the $\gamma_0$ has only a finite number of self-intersections. Thus, there exists an open
connected subset $I\subset\mathds S^1$ and an open subset $U$ of $M$ such that:
\begin{itemize}
\item $\gamma_0:I\to M$ is injective;
\item $\gamma_0(\theta)\in U$ if and only if $\theta\in I$;
\item $\dot\gamma_0(\theta)$ and $J_0(\theta)$ are linearly independent for all $\theta\in I$.
\end{itemize}
By \cite[Lemma~2.4]{BilJavPic}, given a sufficiently small subinterval $I_0\subset I$, there exists a symmetric
$(0,2)$-tensor $h$ of class $C^k$ on $M$ that has compact support contained in $U$ and that has arbitrarily prescribed
value and covariant derivative in the direction $J_0$ along $\gamma_0\vert_{I_0}$.
In particular, one can choose $h$ vanishing identically along $\gamma_0$, and $\nabla_{J_0(\theta)}h=K_\theta$
a symmetric bilinear form on $T_{\gamma_0(\theta)}M$ vanishing for $\theta$ outside $I_0$ and such that:
\[\int_{I_0}K_\theta\big(\dot\gamma_0(\theta),\dot\gamma_0(\theta)\big)\,\mathrm d\theta\ne0.\]
For such $h$ condition \eqref{eq:transvgeo} holds, and the proof is concluded.
\end{proof}
\subsection{Iterates}
\label{sub:iterate}
We will now study the question of nondegeneracy for iterates and prove the strong version of
the semi-Riemannian bumpy metric theorem. We will follow closely Anosov's approach
to iterates as discussed in \cite{Ano}, with suitable modifications.

Let us fix an auxiliary Riemannian metric $\mathbf g_{\mathrm R}$ on $M$. Given a curve $\gamma\in\Lambda$,
the quantity
\[\mathrm E(\gamma)=\tfrac12\int_{\mathds S^1}\mathbf g_{\mathrm R}(\dot\gamma,\dot\gamma)\,\mathrm d\theta\]
will denote the \emph{total energy} of $\gamma$, and the quantity:
\[\mathrm E_{\mathrm{min}}(\gamma)=\mathrm E(\gamma)\cdot\big\vert\mathrm{stab}(\gamma)\big\vert^{-1},\]
where $\big\vert\mathrm{stab}(\gamma)\big\vert$ is the cardinality of the stabilizer of $\gamma$, will
denote the \emph{minimal energy} of $\gamma$. This is the total energy of the (unique) prime geodesic
of whom $\gamma$ is an iterate; clearly, if $\gamma$ is prime, then $\mathrm E(\gamma)=
\mathrm E_{\mathrm{min}}(\gamma)$. Let us also introduce the following notation: given
real numbers $0<a\le b<+\infty$, set:
\begin{multline*}
\mathcal M(a,b)=\Big\{\mathbf g\in\mathrm{Met}(M,i;k):\text{every closed $\mathbf g$-geodesic $\gamma$}\\
\text{with $\mathrm E_{\mathrm{min}}(\gamma)\le a$ and $\mathrm E(\gamma)\le b$ is nondegenerate}\Big\}.
\end{multline*}
Finally, let us denote by $\mathcal M^\star$ the set:
\[\mathcal M^\star=\Big\{\mathbf g\in\mathrm{Met}(M,i;k):\text{every closed prime $\mathbf g$-geodesic $\gamma$
is nondegenerate}\Big\};\]
by Proposition~\ref{thm:weakbumpy}, $\mathcal M^\star$ is generic in $\mathrm{Met}(M,i;k)$.

Our proof of the semi-Riemannian bumpy metric theorem will be obtained by showing that
for all natural numbers $n\ge1$, the set $\mathcal M(n,n)$ is open and dense in $\mathrm{Met}(M,i;k)$;
note that the set of bumpy metrics coincides with the countable intersection:
\[\bigcap_{n\ge1}\mathcal M(n,n).\]
Recall that $\mathrm{Met}(M,i;k)$ is an open subset of a Banach space, thus it is a Baire space, i.e.,
the intersection of a countable family of dense open subsets is dense.
\smallskip

First, we observe that given positive real numbers $a\le a'$ and $b\le b'$, with $a\le b$ and $a'\le b'$,
then: 
\[\mathcal M(a',b')\subset\mathcal M(a,b).\]

\begin{lem}\label{thm:lemit2}
Given $\mathbf g_0\in\mathrm{Met}(M,i;k)$, then there exists $\overline R>0$ and a neighborhood
$\mathcal U_0$ of $\mathbf g_0$ in $\mathrm{Met}(M,i;k)$ such that, for every $\mathbf g\in\mathcal U_0$,
no nonconstant closed $\mathbf g$-geodesic has image contained in a ball of $\mathbf g_{\mathrm R}$-radius
less than or equal to $\overline R$. In particular, there exists $\overline a>0$ such that
for all $\mathbf g\in\mathcal U_0$ and all prime closed $\mathbf g$-geodesic $\gamma$, it is
$\mathrm E(\gamma)\ge\overline a$.
\end{lem}
\begin{proof}
Given any $p\in M$, there exists an open neighborhood $U_p$ of $p$ in $M$ and
an open neighborhood $\mathcal U_0^p$ of $\mathbf g_0$ in $\mathrm{Met}(M,i;k)$ such that,
for all $\mathbf g\in\mathcal U_0^p$, $U_p$ is contained in a $\mathbf g$-convex neighborhood\footnote{%
A neighborhood of $p$ is a normal neighborhood relatively to $\mathbf g$
if it is the diffeomorphic image through the exponential map of $\mathbf g$ of some star-shaped
open neighborhood of $0$ in $T_pM$. A neighborhood of $p$ is convex if it is a normal neighborhood of
all of its points. Convex neighborhoods of a given point exist for every semi-Riemannian metric $\mathbf g$, and their
size depends continuously on $\mathbf g$ relatively to the $C^2$-topology, see \cite{One83}.}
of $p$.
By compactness, $M$ is covered by a finite union $U_{p_1}\cup\ldots\cup U_{p_N}$ of such
open subsets; let $\overline R$ be the Lebesgue number of this open cover
relatively to the metric induced by $\mathbf g_{\mathrm R}$.
Then, every ball of $\mathbf g_{\mathrm R}$-radius less than or equal to $\overline R$ is
contained in some $U_{p_i}$, and thus it cannot contain any non trivial
closed $\mathbf g$-geodesic for any $\mathbf g\in\mathcal U_0=\bigcap_{i=1}^N\mathcal U_0^{p_i}$.
This concludes the proof.
\end{proof}
\begin{lem}\label{thm:lemaconver}
 Let $\{\mathbf g_n\}$ be a sequence of metrics in $\mathrm{Met}(M,i;k)$ converging to $\mathbf g_\infty\in\mathrm{Met}(M,i;k)$ and let $\{\gamma_n\}$ be  curves in $M$ such that for every $n\in\N$, $\gamma_n$ is a degenerate geodesic of $\mathbf g_n$  and there exists a positive number $b$ such that, $E(\gamma_n)\leq b$. Then there exists a subsequence of $\{\gamma_n\}$ that converges to a non constant degenerate geodesic $\gamma_\infty$ of ${\mathbf g}_\infty$.
\end{lem}
\begin{proof}
Since $\mathrm E(\gamma_n)\le b$, then there exists $t_n\in[0,1]$ such that
$\Vert\dot\gamma_n(t_n)\Vert\le b$ for all $n$; up to subsequences, we can assume that $t_n\to t_\infty\in[0,1]$
and that $\dot\gamma_n(t_n)\to v\in ¨T_{p_\infty}M$ as $n\to\infty$, with $p_\infty=\lim_{n\to\infty}\gamma_n(t_\infty)$.
Now, by continuous dependence results on ODE's, the solution $\gamma_\infty$ of the initial value problem
$\frac{\mathrm D^\infty}{\mathrm dt}\dot\gamma=0$, $\gamma(t_\infty)=p_\infty$, $\dot\gamma(t_\infty)=v$,
where $\mathrm D^\infty$ is the covariant derivative of the Levi--Civita connection of $\mathbf g_\infty$,
is $C^2$-limit of the sequence $\gamma_n$ and of course $\gamma_\infty$ is a $\mathbf g_\infty$-closed geodesic with $\mathrm E(\gamma_\infty)\le b$.
Moreover, $\gamma_\infty$ is not constant; namely, if it were, there would be nontrivial closed geodesics
relatively to metrics arbitrarily near $\mathbf g_\infty$ whose images lie in balls of
$\mathbf g_{\mathrm R}$-radius arbitrary small, which contradicts Lemma \ref{thm:lemit2}.

Finally, $\gamma_\infty$ is degenerate. In order to prove the assertion, it suffices to use a continuous dependence
argument for the Jacobi equation along $\gamma_n$. More precisely, let $J_n$ be a periodic Jacobi field along
$\gamma_n$ which is not a multiple of the tangent field $\dot\gamma_n$. By adding to $J_n$ a suitable multiple
of $\dot\gamma_n$, one can assume that $J_n(0)$ is $\mathbf g_{\mathrm R}$-orthogonal to
$\dot\gamma_n(0)$. Moreover, after a suitable normalization, one can assume that:
\[\max\big\{\Vert J_n(0)\Vert,\ \Vert J_n'(0)\Vert\big\}=1;\]
here $J_n'$ denotes the covariant derivative of $J_n$ along $\gamma_n$ relative to the Levi--Civita connection
of $\mathbf g_n$. Then, up to subsequences there exist the limits $\lim\limits_{n\to\infty}J_n(0)=v\in T_{\gamma_\infty(0)}M$
and $\lim\limits_{n\to\infty}J_n'(0)=w\in T_{\gamma_\infty(0)}M$; by continuity
$v$ is $\mathbf g_{\mathrm R}$-orthogonal to $\dot\gamma_\infty(0)$ and
\begin{equation}\label{eq:maxvw}
\max\big\{\Vert v\Vert,\ \Vert w\Vert\big\}=1.
\end{equation}
The Jacobi field $J_\infty$ along $\gamma_\infty$ satisfying the initial conditions $J_\infty(0)=v$ and
$J_\infty'(0)=w$ is uniform (in fact, $C^2$) limit of the Jacobi fields $J_n$, and thus it is periodic.
It is not a multiple of the tangent field $\dot\gamma_\infty$. Namely, if it were, since $v$ is
$\mathbf g_{\mathrm R}$-orthogonal to $\dot\gamma_\infty(0)$, then it would be $v=0$. Moreover, if $J_n$
were a multiple of $\dot\gamma_\infty$, then it would be $w=0$. This contradicts \eqref{eq:maxvw},
and proves that $\gamma_\infty$ is degenerate.
\end{proof}

\begin{lem}\label{thm:lemit1}
For all $0<a\leq b$, $\mathcal M(a,b)$ is open in $\mathrm{Met}(M,i;k)$.
\end{lem}
\begin{proof}
Let us show that the complementary $\mathrm{Met}(M,i;k)\setminus\mathcal M(a,b)$ is closed.
Assume that $\mathbf g_n$ is a sequence in $\mathrm{Met}(M,i;k)\setminus\mathcal M(a,b)$
that converges to some $\mathbf g_\infty$ in $\mathrm{Met}(M,i;k)$. By definition, every $\mathbf g_n$
has a non trivial degenerate closed geodesic $\gamma_n$ with $\mathrm E(\gamma_n)\le b$ and
$\mathrm E_{\mathrm{min}}(\gamma_n)\le a$. By Lemma ~\ref{thm:lemaconver}, there exists a subsequence of $\gamma_n$ that converges to a non constant $\mathbf g_\infty$-degenerate closed geodesic $\gamma_\infty$.
By Lemma~\ref{thm:lemit2}, there exists $\overline a>0$
such that, for $n$ sufficiently large, $\mathrm E(\gamma_n)\big\vert\mathrm{stab}(\gamma_n)\big\vert^{-1}\ge\overline a$;
namely, $\mathrm E(\gamma_n)\big\vert\mathrm{stab}(\gamma_n)\big\vert^{-1}$ is the total energy of a nontrivial
prime closed
geodesic relatively to a metric near $\mathbf g_\infty$. This implies, first, that $\big\vert\mathrm{stab}(\gamma_n)\big\vert$
is bounded, so that, up to passing to subsequences, we can assume $N=\big\vert\mathrm{stab}(\gamma_n)\big\vert$ constant.
Second, by pointwise convergence, $\big\vert\mathrm{stab}(\gamma_\infty)\big\vert\ge N$ (take the limit
in the expression $\gamma_n(t+1/N)=\gamma_n(t)$). Thus, $\mathrm E_{\mathrm{min}}(\gamma_\infty)\le a$.
Hence, $\mathbf g_\infty\in\mathrm{Met}(M,i;k)\setminus\mathcal M(a,b)$,
and the proof is concluded.
\end{proof}
\begin{lem}\label{thm:lemit3}
For all $a>0$, $\mathcal M^\star\cap\mathcal M(a,2a)\subset\mathcal M(2a,2a)$.
\end{lem}
\begin{proof}
Choose $\mathbf g\in\mathcal M^\star\cap\mathcal M(a,2a)$ and let $\gamma$ be a closed
$\mathbf g$-geodesic with $\mathrm E(\gamma)\le 2a$. If $\gamma$ is prime, then it is nondegenerate,
because $\mathbf g\in\mathcal M^\star$. If $\big\vert\mathrm{stab}(\gamma)\big\vert\ge2$, then
$\mathrm E_{\mathrm{min}}(\gamma)\le a$, and thus $\gamma$ is nondegenerate, because
$\mathbf g\in\mathcal M(a,2a)$.
\end{proof}
\begin{lem}\label{thm:lemit4}
For all $a>0$, $\mathcal M\big(\frac32a,\frac32a\big)\cap\mathcal M(a,2a)$ is dense in $\mathcal M(a,2a)$.
\end{lem}
\begin{proof}
Let us show that $\mathcal M^\star\cap\mathcal M(a,2a)$ is contained in
$\mathcal M\big(\frac32a,\frac32a\big)\cap\mathcal M(a,2a)$.
The thesis will then follow from Proposition~\ref{thm:weakbumpy}, which guarantees that
$\mathcal M^\star\cap\mathcal M(a,2a)$ is dense in $\mathcal M(a,2a)$.
Choose $\mathbf g\in\mathcal M^\star\cap\mathcal M(a,2a)$ and let $\gamma$ be a closed
$\mathbf g$-geodesic such that $\mathrm E(\gamma)\le\frac32a$. If $\gamma$ is prime, then
it is nondegenerate because $\mathbf g\in\mathcal M^\star$. If $\big\vert\mathrm{stab}(\gamma)\big\vert\ge2$, then
$\mathrm E_{\mathrm{min}}(\gamma)\le\frac34 a<a$, and thus $\gamma$ is nondegenerate, because
$\mathbf g\in\mathcal M(a,2a)$.
\end{proof}
\begin{lem}\label{thm:lemit5}
For all $a>0$, $\mathcal M(a,2a)$ is dense in $\mathcal M(a,a)$.
\end{lem}
\begin{proof}
Let $\mathbf g_0\in\mathcal M(a,a)$ be fixed and let $\mathcal U$ be an arbitrary open
neighborhood of $\mathbf g_0$ in $\mathcal M(a,a)$. There exists only a finite number of
geometrically distinct\footnote{
Two closed geodesics $\gamma_1,\gamma_2:\mathds S^1\to M$ are \emph{geometrically distinct} if
the sets $\gamma_1(\mathds S^1)$ and $\gamma_2(\mathds S^1)$ are distinct. In particular, geometrically distinct
geodesics belong to different $\mathds S^1$-orbits of $\Lambda$.}
prime closed $\mathbf g_0$-geodesics of energy less than or equal to $a$, say
$\gamma_1$,\ldots,$\gamma_r$, and they are all nondegenerate by assumption.
Namely, if there were infinitely many, then they would accumulate to a necessarily degenerate
closed prime $\mathbf g_0$-geodesic of energy less than or equal to $a$; but such geodesic
does not exist (this can be shown as in the proof of Lemma \ref{thm:lemaconver}).

Now, there exists an open neighborhood $\mathcal U_0\subset\mathcal U$ of $\mathbf g_0$, which is
the domain of smooth functions $\gamma_j:\mathcal U_0\to\Lambda$, $j=1,\ldots,r$, with the following
properties:
\begin{itemize}
\item[(a)] $\gamma_j(\mathbf g)$ is a closed prime nondegenerate $\mathbf g$-geodesic for all $\mathbf g\in\mathcal U_0$;
\item[(b)] given $\mathbf g\in\mathcal U_0$, if $\gamma$ is a closed prime $\mathbf g$-geodesic
near one of the $\gamma_j$'s, then $\gamma$ coincides with $\gamma_j(\mathbf g)$.
\end{itemize}
The reader will find a proof of this claim in Proposition~\ref{thm:prop1}.

It is easy to see that given $\mathbf g$ sufficiently near $\mathbf g_0$, then every
$\mathbf g$-closed geodesics $z$ with $\mathrm E(z)\le a$ coincides with
one of the $\gamma_j(\mathbf g)$'s. Namely, assume that this were not the case; then,
there would exist a sequence $\mathbf g_n$ tending to $\mathbf g_0$ and a sequence
$z_n$ of $\mathbf g_n$-closed geodesics with $\mathrm E(z_n)\le a$
and such that $z_n$ does not coincide with any of the $\gamma_j(\mathbf g_n)$.
By (b) above, $z_n$ must then stay away from some open subset $V$ of $\Lambda$ containing the
$\gamma_j$'s. Arguing as in the proof of Lemma~\ref{thm:lemit1}, one would then obtain
a $C^2$-limit $z_\infty$ of (a suitable subsequence of) $z_n$, which is a $\mathbf g_0$-geodesic
with $\mathrm E(z_\infty)\le a$, and that does not coincide with any of the $\gamma_j$'s.
This is impossible, and our assertion is proved.

Now, we claim that we can find $\mathbf g\in\mathcal U_0$ such that all the $\gamma_j(\mathbf g)$ are nondegenerate,
as well as their two-fold iterates $\gamma_j(\mathbf g)^{(2)}$. A proof of this claim follows easily from
a local perturbation argument, see Corollary~\ref{thm:cor3}. More precisely, the result of Corollary~\ref{thm:cor3}
has to be used repeatedly for each $\gamma_j(\mathbf g)$, $j=1,\ldots,r$; the perturbation at the $(j+1)$-st step
has to be chosen small enough so that $\gamma_1(\mathbf g)$, \dots, $\gamma_j(\mathbf g)$ remain
nondegenerate together with their two-fold coverings. Moreover, as observed above,
the perturbation $\mathbf g$ of $\mathbf g_0$ can be found in such a way that
$\mathbf g$ has no closed geodesic of minimal energy less than or equal to $a$ that does not coincide
with any of the $\gamma_j(\mathbf g_0)$'s.
Then, it follows that $\mathbf g$ belongs to $\mathcal M(a,2a)$, because all its closed geodesics
of minimal energy less than or equal to $a$ and their two-fold coverings are nondegenerate.
Then, $\mathcal U\cap\mathcal M(a,2a)\supset\mathcal U_0\cap\mathcal M(a,2a)\ne\emptyset$, which proves
that $\mathcal M(a,2a)$ is dense in $\mathcal M(a,a)$.
\end{proof}
\begin{cor}\label{thm:lemit6}
For all $a>0$, $\mathcal M\big(\frac32a,\frac32a\big)$ is dense in $\mathcal M(a,a)$.
\end{cor}
\begin{proof}
It follows at once from Lemmas~\ref{thm:lemit4} and \ref{thm:lemit5}.
\end{proof}
\begin{prop}\label{thm:lemit7}
For all $b>a$, $\mathcal M(b,b)$ is dense in $\mathcal M(a,a)$.
\end{prop}
\begin{proof}
An immediate induction argument using Corollary~\ref{thm:lemit6} shows that for all $n\ge1$,
$\mathcal M\big((\frac32)^na,(\frac32)^na\big)$ is dense in $\mathcal M(a,a)$.
Choose $n$ such that $(\frac32)^n>b$; then \[\mathcal M\big((\tfrac32)^na,(\tfrac32)^na\big)\subset\mathcal M(b,b),\]
and therefore $\mathcal M(b,b)$ is dense in $\mathcal M(a,a)$.
\end{proof}
Note that for $b\le a$, $\mathcal M(b,b)$ contains $\mathcal M(a,a)$; thus, for all $a$ and $b$,
$\mathcal M(a,a)\cap\mathcal M(b,b)$ is dense in $\mathcal M(a,a)$.

\begin{teo}[Semi-Riemannian bumpy metric theorem]
\label{thm:finalbumpy}
For all $a>0$, $\mathcal M(a,a)$ is dense in $\mathrm{Met}(M,i;k)$, and so the set of bumpy metrics
$\bigcap_{n\ge1}\mathcal M(n,n)$ is generic in $\mathrm{Met}(M,i;k)$.
\end{teo}
\begin{proof}
Fix $a>0$ and $\mathbf g$ in $\mathrm{Met}(M,i;k)$. By Lemma~\ref{thm:lemit2} there exists $\overline a>0$ such that
all closed $\mathbf g$-geodesic has total energy greater than or equal to $\overline a$.
Thus, $\mathbf g\in\mathcal M\big(\frac{\overline a}2,\frac{\overline a}2\big)$.
Given any neighborhood $\mathcal U$ of $\mathbf g$ in $\mathrm{Met}(M,i;k)$, then, since $\mathcal M\big(\frac{\overline a}2,\frac{\overline a}2\big)$
is open in $\mathrm{Met}(M,i;k)$, by
Proposition~\ref{thm:lemit7}, $\mathcal U\cap\mathcal M\big(\frac{\overline a}2,\frac{\overline a}2\big)\cap
\mathcal M(a,a)$ is not empty. Thus, $\mathcal M(a,a)$ is dense in $\mathrm{Met}(M,i;k)$.
\end{proof}
We have the following immediate corollary:
\begin{cor}
There exists a generic set $\mathcal B$ of semi-Riemannian metrics on $M$ such that, for all $a>0$ and
every $\mathbf g\in\mathcal B$, the number of (geometrically distinct) closed $\mathbf g$-geodesics $\gamma$ in $M$
with $\mathrm E(\gamma)\le a$ is finite.
\end{cor}
\begin{proof}
The set $\mathcal B$ of semi-Riemannian metrics satisfying the hypothesis contains the set of bumpy metrics.
Namely, given $a>0$, choose $n\in\mathds N$ with $n\ge a$. Given $\mathbf g\in\mathcal M(n,n)$, then
there are finitely many orbits $\mathds S^1\cdot\gamma$ of closed $\mathbf g$-geodesic $\gamma$ with
$\mathrm E(\gamma)\le a$, see the proof of Lemma~\ref{thm:lemit5}. Thus,
$\mathcal B\supset\bigcap_{n\ge1}\mathcal M(n,n)$.
\end{proof}
\end{section}

\begin{section}{Further genericity results for semi-Riemannian metrics}
\label{sec:further}
As in the Riemannian case, the semi-Riemannian bumpy metric theorem paves the way to a whole
collection of further genericity results for the geodesic flow. Namely, the closed geodesics
of bumpy metrics are isolated, and genericity of a given property of the geodesic flow can be
established using local perturbations of the metric around each closed geodesic.
In this section we give an idea of this procedure, by establishing a semi-Riemannian
analogue of a genericity result for higher order properties of the linearized Poincar\'e map
of closed geodesics due to Klingenberg and Takens, see \cite{KliTak}.

Local perturbations of metrics around a (closed) geodesic $\gamma$ are usually dealt with
using Fermi coordinates. Such coordinates $(t,x_1,\ldots,x_n)$ are characterized
by the fact that $\gamma$ corresponds to the segment $(t,0,\ldots,0)$, the first
derivatives of the metric coefficients $\mathbf g_{ij}$ vanish along $\gamma$, and
the coordinate fields $\frac\partial{\partial x_i}$ are orthogonal to $\dot\gamma$
along $\gamma$, for all $i=1,\ldots,n$. When the metric $\mathbf g$ is semi-Riemannian,
then this type of coordinates exists only along \emph{non lightlike} geodesics $\gamma$.
Note in fact that the properties of Fermi coordinates imply that $\dot\gamma(t)$ and its
orthogonal space $\dot\gamma(t)^\perp$ generate the entire tangent space $T_{\gamma(t)}M$,
i.e., $\gamma$ is not lightlike.
In order to extend the local perturbation techniques to semi-Riemannian geodesics, we therefore need the following
fact, whose basic idea has already been used implicitly in the proof of Lemma~\ref{thm:lemit5}:
\begin{prop}\label{thm:Laopen}
For all $a>0$, the set:
\begin{multline}
\mathcal L(a)=\Big\{\mathbf g\in\mathrm{Met}(M,i;k):\ \text{every closed $\mathbf g$-geodesic $\gamma$
with $\mathrm E(\gamma)<a$}\\\text{is nondegenerate and non lightlike}\Big\}
\end{multline}
is open and dense in $\mathrm{Met}(M,i;k)$. In particular, the set $\mathcal L=\bigcap_{n\in\mathds N}\mathcal L(n)$
consisting of all metrics without lightlike closed geodesics and without degenerate closed geodesics is generic
in $\mathrm{Met}(M,i;k)$.
\end{prop}
\begin{proof}
Note that $\mathcal L(a)\subset\mathcal M(a,a)$.
The fact that $\mathcal L(a)$ is open follows by the same argument used
in the proof of Lemma~\ref{thm:lemaconver}. Namely, if $\mathbf g_n$ is a sequence
in the complement of $\mathcal L(n)$ converging to $\mathbf g_\infty$ in
$\mathrm{Met}(M,i;k)$, and $\gamma_n$ is a sequence of closed $\mathbf g_n$-geodesics
that are either degenerate or lightlike\footnote{the two possibilities being not mutually exclusive}
with $\mathrm E(\gamma_n)\le a$, then
some subsequence of $\gamma_n$ is $C^2$-convergent to some $\mathbf g_\infty$ closed
geodesic $\gamma_\infty$ with $\mathrm E(\gamma_\infty)\le a$. If infinitely many
$\gamma_n$ are degenerate, then so is $\gamma_\infty$; if infinitely many $\gamma_n$ are
lightlike, then also $\gamma_\infty$ is lightlike. Hence $\mathbf g_\infty\not\in\mathcal L(a)$,
and $\mathcal L(a)$ is open.

In order to show that $\mathcal L(a)$ is dense in $\mathrm{Met}(M,i;k)$, by Theorem~\ref{thm:finalbumpy}
it suffices to show that it is dense in $\mathcal M(a,a)$. To this end, we argue as in the proof
of Lemma~\ref{thm:lemit5}: for $\mathbf g_0\in\mathcal M(a,a)$, there exists an open neighborhood
$\mathcal U_0$ of $\mathbf g_0$ in $\mathrm{Met}(M,i;k)$, a positive integer $r$ and smooth functions
$\gamma_1,\ldots,\gamma_r:\mathcal U_0\to\Lambda$ such that $\gamma_j(\mathbf g)$ is a nondegenerate closed
$\mathbf g$-geodesic for all $\mathbf g\in\mathcal U_0$, and given ${\mathbf g}\in\mathcal U_0$ and
$\gamma$ any closed $\mathbf g$-geodesic with $\mathrm E(\gamma)\le a$, then $\gamma$ must coincide
with one of the $\gamma_j(\mathbf g)$'s. Now, using Proposition~\ref{thm:prop2}, it is easy to see
that arbitrarily close to $\mathbf g_0$, one can find metrics $\mathbf g\in\mathcal U_0$
such that none of the $\gamma_j(\mathbf g)$ is lightlike, for $j=1,\ldots,r$\footnote{%
This follows easily from the fact that, by continuity, the set of $\mathbf g\in\mathcal U_0$ such that
$\gamma_j(\mathbf g)$ is not lightlike is open in $\mathcal U_0$ for all $j$.}.
Now, such metrics $\mathbf g$ belong to $\mathcal L(a)$, hence $\mathcal L(a)$ is dense in $\mathcal M(a,a)$.
\end{proof}
In order to state our final result, let us recall the following notations from \cite{KliTak}.
Given positive integers $k$ and $n$, denote by $(\R^{2n},\omega_0)$ the standard symplectic space,
and let $J_{\mathrm s}^k(2n)$ denote the space of $k$-jets at $0$ of symplectic diffeomorphisms of
(open neighborhood of $0$ in) $(\R^{2n},\omega_0)$ that fix $0$. This is a vector space with a product structure, and its invertible
elements form a Lie group. A subset $Q\subset J_{\mathrm s}^k(2n)$ is said to be \emph{invariant}
if $\sigma Q\sigma^{-1}=Q$ for all invertible element $\sigma$ in $J_{\mathrm s}^k(2n)$.

Let $\mathbf g$ be a semi-Riemannian metric  on $M$ and
let $\gamma$ be a closed non lightlike $\mathbf g$-geodesic.
Then, given a small hypersurface $\Sigma\subset M$ passing through $\gamma(0)$
and transverse to $\gamma$, in total analogy with the Riemannian case
one can define the Poincar\'e map $\mathfrak P_{\gamma,\Sigma}:\Sigma^*\to\Sigma^*$,
where $\Sigma^*\subset TM^*$ is given by
\[\Big\{p\in \bigcup_{q\in\Sigma} T_qM^*:\mathbf g^{-1}(p,p)=\mathbf g\big(\dot\gamma(0),\dot\gamma(0)\big)\Big\};\]
here $\mathbf g^{-1}$ denotes the $(2,0)$-tensor on $M$ induced by $\mathbf g$.
This map preserves the symplectic structure on $\Sigma^*$ induced by the canonical symplectic
form of $TM^*$, and it has $\mathbf g\dot\gamma(0)$ as fixed point. Using symplectic coordinates,
one can think of $\mathfrak P_{\gamma,\Sigma}$ as a symplectic diffeomorphism of an open neighborhood
of $0$ in $(\R^{2n},0)$ that fixes $0$. As in \cite{KliTak}, the fact that the $k$-jet of $\mathfrak P_{\gamma,\Sigma}$
at $\mathbf g\dot\gamma(0)$ belongs to some open invariant subset $Q$ of $J_{\mathrm s}^k(2n)$ does not
depend on the choice of the transverse hypersurface $\Sigma$ nor on the choice of symplectic
coordinates on $\Sigma^*$.
\begin{cor}[semi-Riemannian Klingenberg--Takens genericity]
Let $k\ge1$ be fixed and let $Q$ be a dense open invariant subset of $J_{\mathrm s}^k(2n)$.
Then, for $l>k$, the set $\mathcal M_Q$ of all metrics $\mathbf g\in\mathrm{Met}(M,i;l)$ such that:
\begin{itemize}
\item[(i)] all closed $\mathbf g$-geodesic are non lightlike and nondegenerate;
\item[(ii)] given any closed $\mathbf g$-geodesic $\gamma$, then the $k$-th jet of the Poincar\'e map
$\mathfrak P_{\gamma,\Sigma}$ belongs to $Q$,
\end{itemize}
is generic in $\mathrm{Met}(M,i;l)$.
\end{cor}
\begin{proof}
For all $a>0$, define $\mathcal M_Q(a)$ to be the set of those $\mathbf g\in\mathrm{Met}(M,i;l)$
for which the assumptions (i) and (ii) hold only for those closed $\mathbf g$-geodesics
$\gamma$ with $\mathrm E(\gamma)\le a$. Since $Q$ is open, property (ii) is open relatively to
the $C^l$ topology, for all $l>k$. Thus, repeating the standard openness argument used
in Lemma~\ref{thm:lemit1} and in Proposition~\ref{thm:Laopen},
one proves that $\mathcal M_Q(a)$ is open in $ \mathrm{Met}(M,i;l)$.

Since $\mathcal M_Q(a)\subset\mathcal L(a)$, by Proposition~\ref{thm:Laopen} it suffices to show
that $\mathcal M_Q(a)$ is dense in $\mathcal L(a)$. This follows easily by the same proof of Klingenberg
and Takens for the Riemannian case, that can be carried over to the non lightlike semi-Riemannian case.
This concludes the proof.
\end{proof}

\end{section}

\appendix
\begin{section}{Smooth dependence of closed geodesics on the metric}
\label{sec:appA}
Using the notion of good submanifolds for the free loop space $\Lambda$
we will now give a formal proof of the following result:
\begin{prop}\label{thm:prop1}
Let $M$ be a compact manifold, let $\mathbf g_0$ be a semi-Riemannian metric tensor on $M$ of index $i\in\{0,\ldots,\mathrm{dim}(M)\}$
and of class $C^k$, $k\ge2$, and let $\gamma_0\in\Lambda$ be a nondegenerate closed $\mathbf g_0$-geodesic
in $M$. Then, there exists a neighborhood $\mathcal U_0$ of $\mathbf g_0$ in $\mathrm{Met}(M,i;k)$ and a smooth
function $\gamma:\mathcal U_0\to\Lambda$ such that $\gamma(\mathbf g)$ is a $\mathbf g$-geodesic for all
$\mathbf g\in\mathcal U_0$. Moreover, for $\mathbf g$ in $\mathcal U_0$, $\gamma(\mathbf g)$ is the
unique closed $\mathbf g$-geodesic near $\gamma_0$, and it is nondegenerate.
\end{prop}
\begin{proof}
By Proposition~\ref{thm:exSngeo}, there exists a good submanifold $S_0$ of $\Lambda$ through $\gamma_0$;
every metric $\mathbf g\in\mathrm{Met}(M,i;k)$ admits a (nondegenerate) closed geodesic
near $\gamma_0$ if and only if the functional $f_{\mathbf g}\vert_{S_0}$ has a (nondegenerate) critical
point in $S_0$. Consider the smooth map $f:\mathrm{Met}(M,i;k)\times S_0\to\R$ defined by
$f(\mathbf g,\gamma)=f_{\mathbf g}(\gamma)$, and the partial derivative
\[\frac{\partial f}{\partial\gamma}:\mathrm{Met}(M,i;k)\times S_0\longrightarrow TS_0^*.\]
The condition that $\gamma_0$ is a nondegenerate $\mathbf g_0$-geodesic says that
$\frac{\partial f}{\partial\gamma}(\mathbf g_0,\gamma_0)$ belongs to the zero section
$\mathbf 0$ of $TS_0^*$, and that $\frac{\partial f}{\partial\gamma}$ is transversal at $(\mathbf g_0,\gamma_0)$
to $\mathbf 0$. By the implicit function theorem, near $(\mathbf g_0,\gamma_0)$ the inverse image
$\frac{\partial f}{\partial\gamma}^{-1}(\mathbf 0)$ is the graph of a smooth function $\mathbf g\mapsto\gamma(\mathbf g)$.
By continuity, for $\mathbf g$ near $\mathbf g_0$, $\gamma(\mathbf g)$ is nondegenerate.
\end{proof}
Let us show that, around a nondegenerate lightlike closed geodesics, one can find
closed geodesics of arbitrary causal character of nearby metrics.
\begin{prop}\label{thm:prop2}
Let $\gamma_0\in\Lambda$ be a nondegenerate lightlike $\mathbf g_0$-geodesic. Then,
arbitrarily near $\mathbf g_0$ in $\mathrm{Met}(M,i;k)$ one can find metrics $\widetilde{\mathbf g}$
having spacelike and metrics having timelike closed nondegenerate geodesics near $\gamma_0$.
Such metrics $\widetilde{\mathbf g}$ can be found in such a way that the difference $\mathbf g_0-\widetilde{\mathbf g}$
vanishes outside an arbitrarily prescribed open subset $U$ of $M$ containing the image of $\gamma_0$.
\end{prop}
\begin{proof}
Using the terminology above, the statement can be rephrased as follows: the function $\phi:\mathcal U_0\to\R$
defined by
$\phi(\mathbf g)=f\big(\mathbf g,\gamma(\mathbf g)\big)$ must change
sign in arbitrary neighborhoods of $\mathbf g_0$. Observe that $\phi(\mathbf g_0)=0$, so that
if $\phi$ does not change sign in some neighborhood of $\mathbf g_0$, then
$\mathbf g_0$ would be a local extremum
of $\phi$. In this case, it would be $\mathrm d\phi(\mathbf g_0)\mathbf h=0$ for all $\mathbf h$ symmetric $(0,2)$-tensor
of class $C^k$ on $M$. One computes easily:
\[\mathrm d\phi(\mathbf g_0)\mathbf h=\frac{\partial f}{\partial g}(\mathbf g_0,\gamma_0)\mathbf h+\frac{\partial f}{\partial\gamma}(\mathbf g_0,\gamma_0)
\circ\mathrm d\gamma(\mathbf g_0)\mathbf h=\frac{\partial f}{\partial g}(\mathbf g_0,\gamma_0)\mathbf h=
\tfrac12\int_{\mathds S^1}\mathbf h(\dot\gamma_0,\dot\gamma_0)\,\mathrm d\theta.\]
However, the integral on the right hand side in the above equality cannot vanish for all $\mathbf h$;
for instance, if $\mathbf h$ is everywhere positive definite, i.e., a Riemannian metric tensor on $M$, then
such quantity is strictly positive. This shows that arbitrary neighborhoods of $\mathbf g_0$ contain metrics
with timelike and metrics with spacelike closed geodesics near $\gamma_0$. Again, nondegeneracy follows from continuity.

Now, assume that $\widetilde{\mathbf g}$ is one such metric; by continuity,
one can assume that the difference $\mathbf h=\mathbf g_0-\widetilde{\mathbf g}$ is small
enough so that for all $t\in[0,1]$, the sum $\mathbf g_0+t\mathbf h$ is nondegenerate on $M$, i.e.,
a semi-Riemannian metric tensor of index $i$ on $M$.
If $U$ is any open subset of $M$ containing the image of $\gamma_0$, let $\varphi:M\to[0,1]$ be a smooth
function that is identically equal to $1$ near the image of $\gamma_0$ and that vanishes outside $U$.
The metric $\widetilde{\mathbf g}=\mathbf g_0+\varphi\cdot\mathbf h$ satisfies the thesis of the Proposition, and
it coincides with $\mathbf g_0$ outside $U$.
\end{proof}
\begin{cor}\label{thm:cor3}
Let $\gamma_0$ be an arbitrary closed prime $\mathbf g_0$ geodesic in $M$, and let $U$ be an arbitrary
open subset of $M$ containing the image of  $\gamma_0$. Then, arbitrarily near $\mathbf g_0$ in
$\mathrm{Met}(M,i;k)$ one can find a metric $\mathbf g$ with the following properties:
\begin{itemize}
\item[(a)] the difference $\mathbf g-\mathbf g_0$ is a tensor having support in $U$;
\item[(b)] the (unique) closed prime $\mathbf g$-geodesic $\gamma$ near $\gamma_0$ (given by $\gamma(\mathbf g)$, as
in Proposition~\ref{thm:prop1}) is nondegenerate and its two-fold covering $\gamma^{(2)}$ is also nondegenerate.
\end{itemize}
\end{cor}
\begin{proof}
In the Riemannian case, the result (in fact, a more general result on the linearized Poincar\'e map
of $\gamma_0$) is proven in \cite{Kli}, see Proposition~3.3.7, page 108, or \cite{KliTak}.
Note that Klingenberg's result ensures that in the perturbed metric $\mathbf g$, $\gamma_0$ remains a geodesic,
i.e., $\gamma(\mathbf g)=\gamma_0$. The proof of \cite[Proposition~3.3.7]{Kli} employs only symplectic arguments,
not using the positive definite character of the metric.
Thus it carries over to the semi-Riemannian case, except for \emph{one point}; namely, in the use of Fermi coordinates
along $\gamma_0$, it is used the fact that at points $\gamma_0(t)$ along $\gamma_0$, the tangent space
$T_{\gamma_0(t)}M$ is spanned by the tangent vector $\dot\gamma_0(t)$ and its orthogonal space
$\dot\gamma(t)^\perp$. In the general semi-Riemannian case, this fails to be true exactly when
$\gamma_0$ is lightlike.  However, under these circumstances, Proposition~\ref{thm:prop2}
says that one can first perturb the metric $\mathbf g_0$ to a new metric $\widetilde{\mathbf g}$ arbitrarily
near $\mathbf g_0$, that coincides with $\mathbf g_0$ outside $U$, and such that
$\gamma=\gamma(\widetilde{\mathbf g})$ has image contained in  $U$ and it is not lightlike.
Now, Klingenberg's perturbation argument can be applied to the triple $(\widetilde{\mathbf g},\gamma,U)$,
yielding a new metric having $\gamma$ and its two-fold covering $\gamma^{(2)}$ as nondegenerate geodesics.
\end{proof}
\end{section}
\begin{section}{Genericity in the $C^\infty$-topology}
\label{sec:appB}
The genericity of semi-Riemannian bumpy metrics can be proved also relatively to the $C^\infty$-topology.
Note that such topology does not descend from a Banach space structure on the set of symmetric
tensors on $M$, and thus the abstract genericity result of Theorem~\ref{thm:abstrgenericitysmooth}
cannot be applied.
The argument is standard, and it is basically contained in \cite{BilJavPic}; we will repeat it here
for the reader's convenience.

Let us define the subsets
\begin{multline*}
\mathrm{Met}^*_N(M,i;k)=\big\{\mathbf g\in \mathrm{Met}(M,i;k):\text{all closed $\mathbf g$-geodesics $\gamma$
with $E(\gamma)\leq N$}\\ \text{are nondegenerate}\big\}
\end{multline*}
and
\begin{multline*}
\mathrm{Met}^*(M,i;k)=\big\{\mathbf g\in \mathrm{Met}(M,i;k):\text{all closed $\mathbf g$-geodesics
are nondegenerate}\big\}.
\end{multline*}
Clearly $\mathrm{Met}^*(M,i;\infty)=\cap_{N=1}^\infty \mathrm{Met}^*_N(M,i;\infty)$, and since
$\mathrm{Met}^*(M,i;\infty)$ is a Baire space, it is enough to prove that every $\mathrm{Met}^*_N(M,i;\infty)$
is open and dense in $\mathrm{Met}(M,i;\infty)$.

To this aim, we first observe that $\mathrm{Met}^*_N(M,i;k)$ is open in
$\mathrm{Met}(M,i;k)$ for $k=2,\ldots,\infty$. This can be seen arguing as in Lemma~\ref{thm:lemit1}.
Let us prove that $\mathrm{Met}^*_N(M,i;\infty)$ is dense in $\mathrm{Met}(M,i;\infty)$.
By Theorem \ref{thm:finalbumpy},  $\mathrm{Met}^*(M,i;k)$ is dense in $\mathrm{Met}(M,i;k)$,
and $\mathrm{Met}^*_N(M,i;k)$ contains $\mathrm{Met}^*(M,i;k)$, thus $\mathrm{Met}^*_N(M,i;k)$ is dense
in $\mathrm{Met}(M,i;k)$ for every $N>0$.
Recall that $\mathrm{Met}(M,i;\infty)$ is dense in $\mathrm{Met}(M,i;k)$. Then,
$\mathrm{Met}(M,i;\infty)\cap \mathrm{Met}_N^*(M,i;k)=\mathrm{Met}^*_N(M,i;\infty)$ is dense in
$\mathrm{Met}(M,i;k)$ for all $k\geq 2$ because it is the intersection of a dense subset with a
dense open subset. This implies that $\mathrm{Met}^*_N(M,i;\infty)$ is dense in $\mathrm{Met}(M,i;\infty)$
and the proof is concluded.
\end{section}


\begin{thebibliography}{99}

\bibitem{Abr}
\textsc{R.~Abraham}, \emph{Bumpy metrics}, in Global Analysis (Proc.
Sympos. Pure
  Math., Vol. XIV, Berkeley, Calif., 1968), Amer. Math. Soc., Providence, R.I.,
  1970, pp.~1--3.

\bibitem{Ano} \textsc{D. V. Anosov}, \emph{Generic properties
of closed geodesics},  Izv.\ Akad.\ Nauk SSSR Ser.\ Mat.\
\textbf{46} (1982), no.\ 4, 675--709, 896.

\bibitem{AbrRob} \textsc{R. Abraham, J. Robbin}, \emph{Transversal mappings and flows},
W. A. Benjamin, New York, 1967.

\bibitem{BeErEa96}
\textsc{J.~K. Beem, P.~E. Ehrlich, and K.~L. Easley}, {\em Global
{L}orentzian
  geometry}, vol.~202 of Monographs and Textbooks in Pure and Applied
  Mathematics, Marcel Dekker Inc., New York, second~ed., 1996.

\bibitem{BilJavPic} \textsc{L. Biliotti, M. A. Javaloyes, P. Piccione},
\emph{Genericity of nondegenerate critical points and Morse geodesic functionals}, preprint 2008, to appear
in Indiana University Math.\ Journal.

\bibitem{Chi80} \textsc{D. Chillingworth}, \emph{A global genericity theorem for bifurcations in variational
problems}, J. Funct.\ Anal.\ \textbf{35} (1980), 251--278.

\bibitem{ContPat} \textsc{G. Contreras-Barandiar\'an, G. Paternain}, \emph{Genericity of geodesic flows with
positive topological entropy on $S^2$}, J. Diff.\ Geom.\ \textbf{61} (2002), 1--49.

\bibitem{GiaGiaPicCMP} \textsc{R. Giamb\`o, F. Giannoni, P. Piccione}, \emph{ Genericity of Nondegeneracy for Light Rays in Stationary Spacetimes},
Comm.\ Math.\ Phys.\ \textbf{287} (2009), Number 3, 903--923.

\bibitem{GonMir} \textsc{J. A. Gon\c calves Miranda}, \emph{Generic properties for magnetic flows on surfaces}, Nonlinearity \textbf{19}
(2006), 1849--1874.

\bibitem{GreHalVan} \textsc{W. Greub, S. Halperin, R. Vanstone},
\emph{Connections, curvature, and cohomology} Vol. I, II, III:
Pure and Applied Mathematics, Vol. 47, 47-II, 47-III. Academic Press, New York-London, 1972, 1973, 1976.

\bibitem{Kli}
\textsc{W.~Klingenberg}, {\em Lectures on closed geodesics},
Springer-Verlag,
  Berlin, 1978.
\newblock Grundlehren der Mathematischen Wissenschaften, Vol. 230.

\bibitem{KliTak}
\textsc{W.~Klingenberg and F.~Takens}, {\em Generic properties of
geodesic flows},
  Math.\ Ann.\ \textbf{197} (1972), pp.~323--334.

\bibitem{MeyPal} \textsc{K. R. Meyer, J.  Palmore},
\emph{A generic phenomenon in conservative Hamiltonian systems},
1970 Global Analysis (Proc.\ Sympos.\ Pure Math., Vol.\ XIV,
Berkeley, Calif., 1968) pp.\ 185--189.

\bibitem{One83}
\textsc{B.~O'Neill}, {\em Semi-{R}iemannian geometry}, vol.~103 of
Pure and
  Applied Mathematics, Academic Press Inc. [Harcourt Brace Jovanovich
  Publishers], New York, 1983.
\newblock With applications to relativity.

  \bibitem{Sma} \textsc{S. Smale}, \emph{An infinite dimensional version of Sard's theorem},
  Amer.\ J.\ Math.\ \textbf{87} (1965), 861--866.

\bibitem{Whi}
\textsc{B.~White}, {\em The space of minimal submanifolds for
varying {R}iemannian
  metrics}, Indiana Univ.\ Math.\ J.\ \textbf{40} (1991), pp.~161--200.

\end{thebibliography}
\end{document}